\documentclass[a4paper]{amsart}

\usepackage{amssymb,amscd}
\usepackage{mathrsfs} 
\usepackage[all]{xy}

\usepackage[
backref,
pdfauthor={Han Wu},  % add other authors
pdftitle={Tate-Shafarevich groups of elliptic curves with nontrivial 2-torsion subgroups},
]{hyperref}
\usepackage[alphabetic,backrefs,lite]{amsrefs}  % for bibliography, can choose nobysame

\theoremstyle{plain}
\theoremstyle{definition}

\newtheorem{theorem}{Theorem}[section]
\newtheorem{thm}{Theorem}[section]
\newtheorem{lemma}[theorem]{Lemma}

\newtheorem{cor}[theorem]{Corollary}

\newtheorem{prop}[theorem]{Proposition}

\theoremstyle{definition}

\theoremstyle{remark}
\newtheorem{remark}[theorem]{Remark}

\renewcommand{\AA}{\mathbb{A}}

\newcommand{\QQ}{\mathbb{Q}}

\newcommand{\PP}{\mathbb{P}}

\newcommand{\Acal}{{\mathcal A}}

\DeclareMathOperator{\Ker}{Ker}

\DeclareMathOperator{\Res}{Res}

\usepackage[OT2,T1]{fontenc}
\DeclareSymbolFont{cyrletters}{OT2}{wncyr}{m}{n}
\DeclareMathSymbol{\Sha}{\mathalpha}{cyrletters}{"58}
\newcommand{\defi}[1]{\textsf{#1}} % for defined terms

\makeatletter
\g@addto@macro\bfseries{\boldmath}  % This makes math in section titles bold.
\makeatother

\begin{document}
	
	\begin{title}
		{Tate-Shafarevich groups of elliptic curves with nontrivial 2-torsion subgroups}  %\'etale
	\end{title}
	\author{Han Wu}
	\address{University of Science and Technology of China,
		School of Mathematical Sciences,
		No.96, JinZhai Road, Baohe District, Hefei,
		Anhui, 230026. P.R.China.}
	\email{wuhan90@mail.ustc.edu.cn}
	\date{}
	%\thanks{The author was partially supported by USTC}
	% 11G05, , 14H25, 14H52, 14K15, 14J30

	%\thanks{The authors were partially supported by University of Science and Technology of China}
	%\thanks{\textit{MSC 2010} : 11G35 14G05  14G25 14J20}

	% % % ----------------------------------------------------------------------

	% % % ----------------------------------------------------------------------

	\begin{abstract} 
		%In this paper, we discuss the local-global principle of a class of genus one curves. 
		For any number field, 
		%Given a number field of odd degree over the rational number field, 
		we prove that there exists an elliptic curve defined over this field such that its Shafarevich-Tate group has a nontrivial 2-torsion subgroup.
		%we construct a genus one curve over this field, which violates the local-global principle. %The class of this curve gives a nontrivial $2$-torsion element of the Tate-Shafervich group of its Jacobian. %For number fields of even degree over the rational number field, we also can construct genus one curves case by case. 
		%Then we use an example to illustrate this construction.
	\end{abstract} 
	
	\maketitle

	%	\tableofcontents
	
	\section{Introduction}
	
	Throughout this paper, let $K$ be a number field, and let 
	$\Omega_K$ be the set of all nontrivial places of $K.$ For each $v\in \Omega_K,$ let $K_v$ be the completion of $K$ at $v.$ Let $\AA_K$ be the ring of ad\`eles of $K.$ We fix an algebraic closure $\overline{K}$ of $K.$ Let $X$ be a proper algebraic variety defined over $K.$ Let $\overline{X}$ be the base change of $X$ to $\overline{K}.$

	We say that $X$ violates the \defi{local-global principle} if $X(K_v)\neq\emptyset$ for all $v\in\Omega_K,$ whereas $X(K)=\emptyset.$ 
	
	As a consequence of the Hasse-Minkowski theorem for quadratic forms, the local-global principle holds for smooth, projective and geometrically connected curves of genus-0 over every number field. The first examples of varieties violating the local-global principle are due to Lind \cite{Li40} and independently, but a bit later, to  Reichardt \cite{Re42}. For example, they proved that the genus-$1$ curve, given by the smooth
	projective model of plane curve $2y^2 = 1-17x^4$ over $\QQ,$ violating the local-global principle. Shortly thereafter, Selmer \cite{Se51} gave many genus-$1$ curves violating the local-global principle, of which the simplest one is defined over $\QQ$ by
	$3w_0^3+4w_1^3+5w_2^3=0$ in $\PP^2$ with homogeneous coordinates $(w_0:w_1:w_2).$
	Poonen \cite{Po10a} proved that there exist curves over every number field violating the local-global
	principle. Clark \cite[Section 5 Conjecture 1]{Cl09} conjectured that genus-$1$ curve will suffice, i.e. for any number field, there exists a genus-$1$ curve over this field violating the local-global principle. The author \cite{Wu22c} gave an explicit construction to prove that Clark's conjecture \cite[Section 5 Conjecture 1]{Cl09} holds for any number field not containing $\QQ(i).$ 
	
	Our goal is to prove that Clark's conjecture \cite[Section 5 Conjecture 1]{Cl09} holds for any number field. More exactly, we will prove the following theorem.
	
	\begin{thm}[Theorem \ref{thm: main theorem}]
		For any number field $K,$ there exists an elliptic curve $E$ defined over $K$ such that $\Sha (K,E)[2]\neq 0.$ Here  $\Sha (K,E)[2]$ is the $2$-torsion subgroup of the Shafarevich-Tate group of $E.$
	\end{thm} 
	The way to prove this theorem is to prove the existence of some genus-$1$ curve violating the local-global principle, and this curve has 
	a $0$-cycle of degree $4.$ In order to find this curve, we start with a smooth intersection of two quadrics in $\PP^4,$ which is a del Pezzo surfaces of degree $4,$ violating the local-global principle in Section \ref{section Del Pezzo surfaces}. Then we analysis intersections of this surface with hyperplanes in $\PP^4,$ and find a pencil in this surface in Section \ref{section: Projectively dual varieties of del Pezzo surfaces}. The genus-$1$ curve needed is a intersection of this surface with a hyperplane in the chosen pencil. The existence of local points of this curve is from the fibration method, cf. Section \ref{section: main result}. Then the Jacobian of the chosen curve will meet the needs of Theorem \ref{thm: main theorem}.
	
	%	\section{Notation}
	%Given a number field $K,$ let $\Ocal_K$ be the ring of its integers, and let $\Omega_K$ be the set of all its nontrivial places. Let $\infty_K\subset \Omega_K$ be the subset of all archimedean places, and let $2_K\subset \Omega_K$ be the subset of all $2$-adic places. Let $\infty_K^r\subset \infty_K$ be the subset of all real places, and let $\Omega_K^f=\Omega_K\backslash \infty_K$ be the set of all finite places of $K.$ Let $K_v$ be the completion of $K$ at  $v\in \Omega_K.$ %and let $\AA_K$ be the ring of ad\`eles of $K.$ 	
	%%If an element $a\in \Ocal_K$ is a prime element, we denote its prime ideal by $\pfr_a$ and its associated valuation by $v_a.$
	% For $v\in \Omega_K^f,$  let $\FF_v$ be its residue field. %and $\pfr_v$ be the prime ideal associated to $v.$  
	%We say that an element is a \defi{prime element}, if the ideal generated by this element is a prime ideal.

	\section{Del Pezzo surfaces of degree $4$}\label{section Del Pezzo surfaces}
	
	In this section, we briefly recall some facts on Del Pezzo surfaces of degree $4,$ for simplicity, over a number field.  We refer to \cite{Ko99}, \cite{Po17} and \cite{KST89} for this topic.
	
	A \defi{del Pezzo surface} is a smooth, projective, and geometrically connected surface such that its anticanonical divisor is ample. For a del Pezzo surface $X,$ let $K_X$ be the its canonical divisor. The self-intersection number $K_X^2$ on $X$ defines the \defi{degree} of $X,$ denoted by $d_X.$ Since $-K_X$ is ample, the number $K_X^2=(-K_X)^2$ is positive. Hence the degree $d_X$ is a positive integer. The classification of del Pezzo surfaces (\cite[Theoreom 9.4.4]{Po17}) implies that del Pezzo surfaces are geometrically rational. When $d_X=4,$ by \cite[Proposition III 3.4.3]{Ko99}, the anticanonical divisor $-K_X$ is very ample, so it embeds $X$ as a degree $4$ surface in $\PP^4.$ Geometrically,  del Pezzo surfaces of degree $4$ are the blowing up of $\PP^2$ at $5$ points in general position. 
	
	The following lemma states the equivalence between a del Pezzo surface of degree $4$ and a smooth intersection of two quadrics in $\PP^4.$
	\begin{lemma}\label{lemma: del Pezzo surface and intersection of two quadrics}
		Let $X$ be a $K$-scheme. Then $X$ is a del Pezzo surface of degree $4$ if and only if $X$ can be embedded in $\PP^4$ as a smooth intersection of two quadrics.
	\end{lemma}
	\begin{proof}
		By \cite[Theorem III 3.5.4]{Ko99}, a del Pezzo surface of degree $4$ embedding in $\PP^4$ by $-K_X,$ is a intersection of two quadrics. So $X$ can be embedded in $\PP^4$ as a smooth intersection of two quadrics.
		Conversely, since every irreducible component of a intersection of two different quadrics in $\PP^4$ is of codimension $1$ or $2,$ by \cite[Theorem 7.2]{Ha97}, $X$ is  geometrically connected. Hence $X$ is a smooth, projective, and geometrically connected surface. According to the adjunction formula, a smooth intersection of two quadrics in $\PP^4$ gives a 
		del Pezzo surface of degree $4.$  
	\end{proof}

	The local-global principle can fail for degree-$4$ del Pezzo surfaces. Birch and Swinnerton-Dyer \cite[Theorem 3]{BSD75} firstly gave an example of a del Pezzo surface of
	degree $4$ over $\QQ,$ which violates the local-global principle. It is defined by the following two equations: 
	\begin{equation*}
		\begin{cases}
			uv=x^2-5y^2\\ 
			(u+v)(u+2v)=x^2-5z^2
		\end{cases}
	\end{equation*}
	in $\PP^4$ with homogeneous coordinates $(x:y:z:u:v).$ Viray \cite[Theorem 5.0.3]{Vi10} generalized their counterexample to the following theorem.
	\begin{thm}(Compare to \cite[Theorem 5.0.3]{Vi10})\label{theorem: a del Pezzo surface not LGP}
		Let $K$ be a number field. There exists a del Pezzo surface of
		degree $4$ over $K,$  violating the local-global principle.
	\end{thm}
	\begin{remark}
		By choosing constants $a, b, c$ in a given field $K,$ Viray \cite[Theorem 5.0.3]{Vi10} constructed a del Pezzo surface of
		degree $4$ over $K,$  violating the local-global principle. The surface is defined by the following two equations: 
		\begin{equation*}
			\begin{cases}
				uv=x^2-ay^2\\ 
				(u-b^2cv)(cu+(1-b^2c^2)v)=x^2-az^2
			\end{cases}
		\end{equation*} 
		in $\PP^4$ with homogeneous coordinates $(x:y:z:u:v).$ Jahnel and Schindler \cite{JS17}  studied a family of degree-$4$ del Pezzo surfaces over $K,$ defined by the following two equations: 
		\begin{equation*}
			\begin{cases}
				x_0x_1=x_2^2-Dx_3^2\\ 
				(x_0+Ax_1)(x_0+Bx_1)=x_2^2-Dx_4^2
			\end{cases}
		\end{equation*} 
		in $\PP^4$ with homogeneous coordinates $(x_0:x_1:x_2:x_3:x_4)$ and triple parameters $(A,B,D)$ in $K.$
		They proved that there are infinite many del Pezzo surfaces in this family violating the local-global principle. Indeed, they are Zariski dense in the coarse moduli scheme given in \cite[Section 5]{HKT13}.
	\end{remark}
	%generalized their counterexample to any number field, (even any global field of characteristic not $2$). 
	
	\section{Projectively dual varieties of del Pezzo surfaces}\label{section: Projectively dual varieties of del Pezzo surfaces}
	Let $X$ be a del Pezzo surface over $K.$ By Lemma \ref{lemma: del Pezzo surface and intersection of two quadrics}, we assume that $X=Q_0\cap Q_\infty$ is the intersection of two quadrics $Q_0$ and $Q_\infty$ in $\PP^4.$ 
	Let $(\PP^4)^*$ be the dual projective space of $\PP^4,$ cf. \cite{Te05}. Let $Z_X\subset X\times (\PP^4)^*$ be the algebraic subset defined by $$Z_X=\{(x,H)\in X\times (\PP^4)^*|X\cap H \text{ is singular at } x\}.$$
	Let $X^*\subset (\PP^4)^*$ be the projectively dual of $X.$ Then $X^*=Pr_2(Z_X).$ We have the following commutative diagram:
	$$\xymatrix{
		&Z_X \ar_{pr_1}[ld]	\ar^{pr_2}[d] \ar@{^(->}[r] &  X\times (\PP^4)^*\ar^{Pr_2}[d]\\
		X & X^*\ar@{^(->}[r] & 	(\PP^4)^* 
	}$$
	Since $X$ is a smooth surface, the first projection $pr_1\colon Z_X\to X$ makes $Z_X$ into a $\PP^1$ bundle over $X.$ Hence $Z_X$ is smooth, and $\dim Z_X=\dim X+1=3.$ Since $Z_X$ is smooth, projective and geometrically connected, and $X^*=Pr_2(Z_X),$ $X^*$ is a projective and geometrically integral variety. The following proposition states that the second projection $pr_2$ is a birational map between $Z_X$ and $X^*.$
	\begin{prop}\label{prop: pr_2 is birational map}
		Let $X$ be a smooth intersection of two quadrics $Q_0$ and $Q_\infty$ in $\PP^4_K.$ Then the map $pr_2\colon Z_X\to X^*$ is a birational map.
	\end{prop}
	\begin{proof}
		Since this is a geometric argument, it will be suffice to check that $pr_2\colon \overline{Z_X}\to \overline{X}^*$ is a birational map. By \cite[Lemma 2.7]{Vo03}, we only need to show that $\dim \overline{X}^*=3.$
		By \cite[Proposition 2.1]{Re72}, we assume that $Q_0=\sum\limits_{i=0}\limits^4 x_i^2$ and $Q_\infty=\sum\limits_{i=0}\limits^4 c_ix_i^2$ with $c_i\in \overline{K}$ and $c_i\neq c_j$ for all $0\leq i\neq j\leq 4.$ Let $H\subset \PP^4$ be a hyperplane, also a point in $(\PP^4)^*$ with coordinates $(y_0:\cdots:y_4).$ For some integer $i\in [0,4],$ let $U_i\subset (\PP^4)^*$ be the affine open variety defined by $y_i\neq 0.$ Since $(\PP^4)^*$ is covered by $U_i,$ without loss of generality, we may assume $H\in U_4(\overline{K}),$ and $H$ is defined by $x_4+\sum\limits_{i=0}\limits^3 x_iy_i=0$ and $c_4=0.$ Hence $H\cap \overline{X}$ is isomorphic to the intersection of $Q_0'=\sum\limits_{i=0}\limits^3 x_i^2+(\sum\limits_{i=0}\limits^3 x_iy_i)^2$ and $Q_\infty'=\sum\limits_{i=0}\limits^3 c_ix_i^2$ in $\PP^3$ with coordinates $(x_0:\cdots:x_3).$ Consider the pencil of quadrics: $Q_\lambda'=Q_0'+\lambda Q_\infty',~ \lambda\in \overline{K}.$ Let $P(\lambda)=\det (Q_\lambda').$ Then $P(\lambda)$ is a polynomial in $\lambda$ of degree $4,$ and the leading coefficient is $\prod\limits_{i=0}\limits^3c_i\neq 0.$ By the definition of $\overline{Z_X},$ $H\in \overline{X}^*$ if and only if $H\cap \overline{X}$ is singular.  By \cite[Proposition 2.1]{Re72}, the intersection of $Q_0'$ and $Q_\infty'$ is nonsingular if and only if all roots of $P(\lambda)$ are distinct. Let $P'(\lambda)$ be the derivative of $P(\lambda).$ Then $H\cap \overline{X}$ is singular if and only if the polynomials $P(\lambda)$ and $P'(\lambda)$ have a root in common. Let $\Res(P(\lambda),P'(\lambda))$ be the resultant of $P(\lambda), P'(\lambda).$	By \cite[Chapter IV,  Corollary 8.4]{La02}, the polynomials $P(\lambda)$ and $P'(\lambda)$ have a root in common if and only if $\Res(P(\lambda),P'(\lambda))=0.$
		Then $U_4\cap \overline{X}^*$ is the zero locus of $\Res(P(\lambda),P'(\lambda))$ in $U_4.$ Hence $\dim \overline{X}^*=3.$
	\end{proof}
	With this proposition, we have the following corollary.
	\begin{cor}
		Let $X$ be a smooth intersection of two quadrics $Q_0$ and $Q_\infty$ in $\PP^4_K.$ There exists a Zariski open subset $U\subset X^*$ such that for any hyperplane $H\in U(\overline{K}),$ the intersection $H\cap \overline{X}$ has exactly one singular point, and this singular point is an ordinary double point on $H\cap \overline{X}.$
	\end{cor}\label{corollary: pr_2 is birational map}
	\begin{proof}
		By Proposition \ref{prop: pr_2 is birational map}, there exists a nonempty Zariski subset $U'\subset \overline{X}^*$ such that $pr_2^{-1}(U')\to U'$ is an isomorphism. By the definition of $\overline{Z_X},$ the intersection $H\cap \overline{X}$ has exactly one singular point. By \cite[Chapter XVII 3.7]{DK73}, this singular point is an ordinary double point on $H\cap \overline{X}.$
	\end{proof}
	
	The following lemma will be used to analysis intersections of del Pezzo surfaces with hyperplanes in $\PP^4.$
	
	\begin{lemma}\label{lemma: intersection of H and X is integral}
		Let $Q_0$ and $Q_\infty$ be two quadrics in $\PP^3_{\overline{K}}.$ Let $Q_\lambda=Q_0+\lambda Q_\infty,~ \lambda\in \overline{K}.$ Let $P(\lambda)=\det (Q_\lambda),$ and let $P'(\lambda)$ be the derivative of $P(\lambda).$ Assume that $P(\lambda)$ is of degree $4,$ and the greatest common factor of polynomials $P(\lambda)$ and $P'(\lambda)$ is of degree $1.$ Then the intersection $Q_0\cap Q_\infty$ is integral or has two singular points.
	\end{lemma}
	\begin{proof}
		Let $V$ be a vector space of dimension $4$ over $\overline{K}.$ By $\PP^4\cong \PP(V),$ we view $Q_0,~Q_\infty$ as two quadrics in $\PP(V).$ Since there exists a $1$-$1$ correspondence between quadratic forms and symmetric bilinear forms, we denote both them by the same letter $\phi.$ Let $\phi_0,~\phi_\infty$ be two quadratic forms, which determine quadrics $Q_0$ and $Q_\infty$ respectively.
		Since $P(\lambda)$ is of degree $4,$ the quadratic form $\phi_\infty$ is nondegenerate. Let $\phi_\lambda=\phi_0+\lambda\phi_\infty.$ Since the greatest common factor of polynomials $P(\lambda)$ and $P'(\lambda)$ is of degree $1,$ without loss of generality, we assume $(P(\lambda),P'(\lambda))=\lambda.$ Let $\lambda_1,\lambda_2$ be the other two different nonzero roots of $P(\lambda),$ and $P(\lambda)=c\lambda^2(\lambda-\lambda_1)(\lambda-\lambda_2)$ for some nonzero constant $c\in \overline{K}^\times.$ For $i\in {1,2},$ let $e_i\in \Ker \phi_{\lambda_i}$ be a nonzero vector in $V.$ Then $e_1$ and $e_2$ are orthogonal for $\phi_{\lambda_1}$ and $\phi_{\lambda_2},$ and hence for all $\phi_\lambda.$  	
		Since $P(\lambda)$ has the factor $\lambda^2,$ we have $\Ker \phi_0\neq 0.$ There are the following two cases.
		\begin{enumerate}{
				\item  	Suppose that $\dim \Ker \phi_0=2.$ By the same argument as in the previous sentence, the vector spaces $\Ker \phi_0,$ $\overline{K}e_1$ and $\overline{K}e_2$ are orthogonal for all $\phi_\lambda.$ Hence $V=\Ker \phi_0\oplus \overline{K}e_1\oplus \overline{K}e_2.$ By normalizing $e_1,~e_2$ with respect to $\phi_\infty,$ we carefully choose two linearly independent vectors $e_0^1,e_0^2\in \Ker \phi_0$ so that $\phi_\infty(x_0e_0^1+x_1e_0^2+x_2e_1+x_3e_2)=\sum\limits_{i=0}^3x_i^2.$ And $\phi_0(x_0e_0^1+x_1e_0^2+x_2e_1+x_3e_2)=-\lambda_1x_2^2-\lambda_2x_3^2.$ Hence $Q_0\cap Q_\infty$ has two irreducible component, and has exactly two singular points $(x_0:x_1:x_2:x_3)=(\pm \sqrt{-1}:1:0:0).$
				\item Suppose that $\dim \Ker \phi_0=1.$ Since the quadratic form $\phi_\infty$ is nondegenerate, we choose a basis $\xi_0,\cdots,\xi_3$ such that $\phi_\infty(\sum\limits_{i=0}^3 x_i\xi_i)=\sum\limits_{i=0}^3x_i^2.$ Then $\phi_0(\sum\limits_{i=0}^3 x_i\xi_i)=\sum\limits_{i=0}^3\sum\limits_{j=0}^3 a_{ij}x_ix_j$ for some symmetric matrix $A=(a_{ij})_{0\leq i,j\leq 3}$ with coefficients in $\overline{K}.$ The matrix $A$ gives a linear map $V\to V,$ denoted by $\Acal.$  Let $P_\Acal(\lambda)$ be the characteristic polynomial of $\Acal.$ Then $P_\Acal(\lambda)=\lambda^2(\lambda+\lambda_1)(\lambda+\lambda_2),$ $\Ker\phi_0=\Ker\Acal,$ and $-\lambda_1$ and $-\lambda_2$ are eigenvalues of $\Acal$ belonging to the eigenvectors $e_1$ and $e_2$ respectively. Since $P_\Acal(\lambda)$ has the factor $\lambda^2,$ by Jordan normal form theorem \cite[Theorem 6.2]{La04}, $\dim \Ker\Acal^2=2.$ The symmetric matrix $A^2$ gives a quadratic form $\phi_{A^2}.$ By the same argument applies to quadratic forms $\phi_{A^2}$ and $\phi_\infty,$ the vector spaces $\Ker\Acal^2,$ $\overline{K}e_1$ and $\overline{K}e_2$ are orthogonal for $\phi_{A^2}$ and $\phi_\infty.$ Hence $V=\Ker\Acal^2\oplus \overline{K}e_1\oplus \overline{K}e_2.$ Let $e_0^1\in \Ker\Acal$ be a nonzero vector. Since $\Ker\Acal\subsetneq \Ker\Acal^2,$ there exists a vector $e_0'$ such that $e_0^1=\Acal (e_0').$ Then $\phi_\infty(e_0^1)=\phi_\infty(\Acal (e_0'))=\phi_\infty(e_0',\Acal^2 (e_0'))=0.$ Since $\phi_\infty$ is nondegenerate, we can choose a nonzero vector $e_0^2\in \Ker\Acal^2$ such that $\phi_\infty(e_0^1,e_0^2)=1$ and $\phi_\infty(e_0^2,e_0^2)=0.$ By normalizing $e_1,~e_2$ with respect to $\phi_\infty,$ we have 
				$\phi_\infty(x_0e_0^1+x_1e_0^2+x_2e_1+x_3e_2)=x_0x_1+x_2^2+x_3^2,$ and $\phi_0(x_0e_0^1+x_1e_0^2+x_2e_1+x_3e_2)=c_0x_1^2-\lambda_1x_2^2-\lambda_2x_3^2$ for some nonzero constant $c_0\in \overline{K}^\times.$ Hence $Q_0\cap Q_\infty$ is integral, and has exactly one singular point $(x_0:x_1:x_2:x_3)=(1:0:0:0).$ 
		}\end{enumerate}
		In summary, the intersection $Q_0\cap Q_\infty$ is integral or has two singular points.
	\end{proof}
	\begin{remark}
		One can check in both case that the singular points have ordinary singularity.
	\end{remark}
	
	The following proposition states that intersections of $X$ with hyperplanes in $X^*$ are almost geometrically integral.
	\begin{prop}\label{proposition: geometrically integral intersection}
		Let $X$ be a smooth intersection of two quadrics $Q_0$ and $Q_\infty$ in $\PP^4_K.$ There exists a nonempty Zariski open subset $U\subset X^*$ such that for any hyperplane $H\in U(\overline{K}),$ the intersection $H\cap \overline{X}$ is integral.
	\end{prop}
	\begin{proof}
		Since $X^*$ is geometrically integral, by Galois decent, it will be suffice to find a nonempty Zariski open subset $U\subset \overline{X}^*$ such that for any hyperplane $H\in U(\overline{K}),$ the intersection $H\cap \overline{X}$ is integral.
		By the same argument as in the proof of Proposition \ref{prop: pr_2 is birational map}, we assume that $Q_0=\sum\limits_{i=0}\limits^4 x_i^2$ and $Q_\infty=\sum\limits_{i=0}\limits^3 c_ix_i^2$ with nonzero coefficients $c_i\in \overline{K}$ and $c_i\neq c_j$ for all $0\leq i\neq j\leq 3.$
		By Proposition \ref{prop: pr_2 is birational map}, there exists a nonempty Zariski subset $U'\subset \overline{X}^*$ such that $pr_2^{-1}(U')\to U'$ is an isomorphism. 
		Let $(y_0:\cdots:y_4)$ be the coordinates of $(\PP^4)^*,$ and let $U_4\subset (\PP^4)^*$ be the affine open variety defined by $y_4\neq 0.$ 
		Take a hyperplane $\tilde{H}\in U_4(\overline{K}),$ and $\tilde{H}$ is defined by $x_4+\sum\limits_{i=0}\limits^3 x_iy_i=0.$  Hence $\tilde{H}\cap \overline{X}$ is isomorphic to the intersection of $Q_0'=\sum\limits_{i=0}\limits^3 x_i^2+(\sum\limits_{i=0}\limits^3 x_iy_i)^2$ and $Q_\infty'=\sum\limits_{i=0}\limits^3 c_ix_i^2$ in $\PP^3$ with coordinates $(x_0:\cdots:x_3).$ Let $Q_\lambda'=Q_0'+\lambda Q_\infty',~ \lambda\in \overline{K}.$ Let $P(\lambda)=\det (Q_\lambda'),$ and let $P'(\lambda)$ be the derivative of $P(\lambda).$ Let $\Res(P(\lambda),P'(\lambda))$ be the resultant of $P(\lambda), P'(\lambda).$
		From the proof of Proposition \ref{prop: pr_2 is birational map}, $U_4\cap \overline{X}^*$ is the zero locus of $\Res(P(\lambda),P'(\lambda))$ in $U_4.$ The condition that the greatest common factor of polynomials $P(\lambda)$ and $P'(\lambda)$ is of degree $1,$ gives an open condition for the zero locus of $\Res(P(\lambda),P'(\lambda))$ in $U_4.$ Let $U''\subset U_4\cap \overline{X}^*$ be a Zariski open subset such that for any $\tilde{H}\in U''(\overline{K}),$ the greatest common factor of polynomials $P(\lambda)$ and $P'(\lambda)$ is of degree $1$. Let $U=U'\cap U''.$ Take an arbitrary $H\in U(\overline{K}).$ Since $pr_2$ is an isomorphism over $U,$ $H\cap \overline{X}$ has exactly one singular point. Because of $H\in U''(K),$ the greatest common factor of polynomials $P(\lambda)$ and $P'(\lambda)$ is of degree $1.$ By Lemma \ref{lemma: intersection of H and X is integral}, the intersection $H\cap \overline{X}$ is integral.
		%the polynomials $P(\lambda)$ and $P'(\lambda)$ have exactly only one root in common.		
		%Shrink $U'$ if necessary, so that $U'\subset U''.$ 
	\end{proof}
	
	Next, we will choose a pencil in a del Pezzo surface $X$ such that for every $H$ in this pencil, the intersection $H\cap X$ is geometrically integral.
	\begin{prop}\label{proposition: existence of a pencil}
		Let $X$ be a smooth intersection of two quadrics $Q_0$ and $Q_\infty$ in $\PP^4_K.$ There exists a line $D$ in $(\PP^4)^*$ such that for every $H$ in this pencil, the intersection $H\cap X$ is geometrically integral.
	\end{prop}
	\begin{proof}
		Since $\dim X^*=3,$ we take a hyperplane $H_0\in (\PP^4)^*\backslash X^*.$ Then $H_0\cap X$ is smooth. By \cite[Chapter III, Corollary 7.9]{Ha97}, $H_0\cap X$ geometrically connected. (The existence of $H_0$ is also from Bertini's theorem \cite[Chapter II, Theorem 8.18]{Ha97}.)
		By Proposition \ref{proposition: geometrically integral intersection}, there exists a Zariski open subset $U\subset X^*$ such that for any hyperplane $H\in U(\overline{K}),$ the intersection $H\cap \overline{X}$ is integral. Since $\dim X^*=3,$ $\dim (X^*\backslash U)$ is at most $2.$ Then the cone over $X^*\backslash U$ with axis $H_0$ has dimension at $3.$ So we can choose a hyperplane $H_\infty$ not in this cone. Let $D\subset (\PP^4)^*$ be the line jointing $H_0$ and $H_\infty.$ Hence $D\cap X^*\subset U.$ By the same argument  as in the previous sentence and Proposition \ref{proposition: geometrically integral intersection}, for all $H\in D(\overline{K}),$ the intersection $H\cap \overline{X}$ is integral.
	\end{proof} 
	\begin{remark}
		The pencil that we chosen in the proof of this proposition, is a so-called Lefschetz pencil in $X.$
	\end{remark}
	
	With the help of Proposition \ref{proposition: existence of a pencil}, we can find a fibration having the following properties.
	\begin{prop}\label{proposition: existence of genus 1 fiberation and every fibre is integral}
		Let $X$ be a del Pezzo surface $X$ over $K.$ Then there exist a smooth, projective and geometrically integral surface $\tilde{X},$ a birational morphism $\tilde{X}\to X,$ and a dominant 
		morphism $\pi\colon \tilde{X}\to \PP^1$ such that every geometric fiber is integral, and the generic fibre of $\pi$ is a curve of genus $1.$
	\end{prop}
	\begin{proof}
		By Lemma \ref{lemma: del Pezzo surface and intersection of two quadrics}, we embed $X$ in $\PP^4$ as a smooth intersection of two quadrics.
		By Proposition \ref{proposition: existence of a pencil}, let $\PP^1\subset (\PP^4)^*$ be a line such that for every $H\in \PP^1(\overline{K})$, the intersection $H\cap X$ is geometrically integral. Let $\tilde{X}\to X$ be the blowing up of $X$ along the intersection $\bigcap\limits_{t\in \PP^1} H_t\cap X,$ which is a reduced $0$-dimensional scheme of degree $4.$ Let $\pi\colon \tilde{X}\to \PP^1$ be from this blowing up. Then the fiber over $t\in \PP^1(\overline{K})$ is $H_t\cap \overline{X},$ which is integral. The generic fibre of $\pi$ is a smooth intersection of two quadrics in $\PP^3.$ Hence it is a curve of genus $1.$
	\end{proof}
	\begin{remark}\label{remark: degree 4 points on fibres}
		In the proof of this theorem, the fiber over
		each point in $\PP^1(K)$ is a intersection of two quadrics in $\PP^3.$ Hence every rational fibre of $\pi$ has a $0$-cycle of degree $4.$
	\end{remark}
	
	\section{Genus-$1$ curves violating the local-global principle}\label{section: main result}
	In this section, we find a genus-$1$ curve violating the local-global principle, and this curve has 
	a $0$-cycle of degree $4.$ 
	\begin{prop}\label{proposition: genus-$1$ curve $C$ violating the local-global principle}
		For any number field $K,$ there exists a genus-$1$ curve $C$ violating the local-global principle. Furthermore, this curve $C$ has a $0$-cycle of degree $4.$
	\end{prop} 
	
	\begin{proof}
		By Theorem \ref{theorem: a del Pezzo surface not LGP}, we choose a del Pezzo surface $X$ of
		degree $4$ over $K,$  violating the local-global principle. We choose a smooth, projective and geometrically integral surface $\tilde{X}$ as in the proof of Proposition \ref{proposition: existence of genus 1 fiberation and every fibre is integral}. Let $\pi\colon \tilde{X}\to \PP^1$ be a dominant morphism such that every geometric fiber is integral, and the generic fibre of $\pi$ is a curve of genus $1.$ Since $X$ violates the local-global principle, $X(\AA_K)\neq \emptyset$ and $X(K)=\emptyset.$ By the implicit function theorem, $\tilde{X}(\AA_K)\neq \emptyset.$ Since $X$ and $\tilde{X}$ are smooth, projective and geometrically integral surface, and they are birationally equivalent, by comparing the exceptional curves between $X$ and $\tilde{X},$ $X(K)=\emptyset$ implies $\tilde{X}(K)=\emptyset.$ Hence $\tilde{X}$ violates the local-global principle. Since every geometric fiber of $\pi$ is integral, by the fibration method \cite[Lemma 3.1]{CP00} (firstly used in \cite{CTSSD87a} and \cite{CTSSD87b}), there is a finite set of places $S\subset \Omega_K$ such that for every place $v\notin S,$ $\pi(X(K_v))=\PP^1(K_v).$ Since the generic fibre of $\pi$ is a curve of genus $1,$ by semicontinuity theorem \cite[Theorem 12.8]{Ha97}, there exists an open Zariski subset $U\in\PP^1$ such that every geometric fiber over points in $U$ is a smooth curve of genus $1.$
		For $v\in S,$ by the implicit function theorem, the set $\pi(X(K_v))\cap U(K_v),$ denoted by $U_v,$ is a nonempty open subset of $U(K_v).$ For the nonempty open subset $\prod\limits_{v\in S}U_v\times \prod\limits_{v\notin S}\PP^1(K_v)\subset \PP^1(\AA_K),$ since $\PP^1$ satisfies weak approximation, we can choose a $P\in (\prod\limits_{v\in S}U_v\times \prod\limits_{v\notin S}\PP^1(K_v))\cap \PP^1(K).$ Let $C=\pi^{-1}(P)$ be the fibre over $P.$ Then $C$ is a genus-$1$ curve violating the local-global principle. By Remark \ref{remark: degree 4 points on fibres}, the curve $C$ has a $0$-cycle of degree $4.$
	\end{proof}
	\begin{remark}
		The method to proof this proposition is mainly from \cite[Subsection 3.1]{Po10a}.
	\end{remark}
	
	Applying Proposition \ref{proposition: genus-$1$ curve $C$ violating the local-global principle}, we get our main theorem.
	\begin{thm}\label{thm: main theorem}
		For any number field $K,$ there exists an elliptic curve $E$ defined over $K$ such that $\Sha (K,E)[2]\neq 0.$ 
	\end{thm} 
	
	\begin{proof}
		By Proposition \ref{proposition: genus-$1$ curve $C$ violating the local-global principle}, we choose a genus-$1$ curve $C$ violating the local-global principle, and $C$ has a $0$-cycle of degree $4.$ Let $E$ be the Jacobian of $C.$ Then $E$ is an elliptic curve. Consider the class $[C]\in H^1(K, E).$ Since $C$ violates the local-global principle, we have $[C]\in \Sha (K,E)$ and $[C]\neq 0.$ 
		Since $C$ has a $0$-cycle of degree $4,$ we have $4[C]= 0$ in $\Sha (K,E).$ Hence $[C]$ or $2[C]$ is a nonzero element of $\Sha (K,E)[2]$. 
	\end{proof}

	\begin{footnotesize}
	\noindent\textbf{Acknowledgements.} The author would like to thank X.Z. Wang, D.S. Wei, and C. Lv for many fruitful discussions. The author is grateful to anonymous referees for their valuable suggestions. The author was partially supported by NSFC Grant No. 12071448.
\end{footnotesize}


\end{document}